\documentclass[sn-mathphys-num]{sn-jnl}

\usepackage{graphicx}%
\usepackage{multirow}%
\usepackage{amsmath,amssymb,amsfonts}%
\usepackage{amsthm}%
\usepackage{mathrsfs}%
\usepackage[title]{appendix}%
\usepackage{xcolor}%
\usepackage{textcomp}%
\usepackage{manyfoot}%
\usepackage{booktabs}%
\usepackage{algorithm}%
\usepackage{algorithmicx}%
\usepackage{algpseudocode}%
\usepackage{listings}%

\theoremstyle{thmstyleone}%
\newtheorem{theorem}{Theorem}
\newtheorem{lemma}[theorem]{Lemma}
\newtheorem{corollary}{Corollary}[theorem]
\theoremstyle{thmstyletwo}%
\newtheorem{remark}{Remark}%

\theoremstyle{thmstylethree}%

\begin{document}

\title[A consistent kinetic Fokker-Planck model for gas mixtures]{A consistent kinetic Fokker-Planck model for gas mixtures}


\author*[1]{\fnm{Marlies} \sur{Pirner}}\email{marlies.pirner@uni-muenster.de}



\affil*[1]{\orgdiv{Institute of Analysis and Numerics}, \orgname{University of Münster}, \orgaddress{\street{Einsteinstr. 62}, \city{Münster}, \postcode{48161},  \country{Germany}}}




\abstract{ We propose a general multi-species Fokker-Planck model. We prove consistency of our model: conservation properties, positivity of all temperatures, H-Theorem and the shape of equilibrium as Maxwell distributions with the same mean velocity and temperature. Moreover, we derive the usual macroscopic equations from the kinetic two-species BGK model and compute explicitly the exchange terms of momentum and energy.}




\maketitle

\section{Introduction}

\textcolor{black} {In this paper, we propose an extension of the Fokker-Planck equation for gas mixtures. It is often also referred as the Lenard-Bernstein \cite{PhysRev.112.1456} or Dougherty model \cite{10.1063/1.2746779} and is used for the description of a collisional plasma \cite{CROUSEILLES2004546,DUCLOUS20095072,TAITANO2021107861}. The Fokker-Planck equation can be derived as an approximation of the Landau-Fokker-Planck equation by replacing one distribution function in the quadratic Landau-Fokker-Planck operator by the equilibrium Maxwell distribution, for more details see \cite{Hu}. It still maintains the same main properties as the Landau-Fokker-Planck and the Boltzmann operator: conservation properties, H-Theorem and the same shape of equilibrium. The Fokker-Planck operator is computationally cheaper than the full Landau-Fokker-Planck or the full Boltzmann operator  and is therefore often used in numerical simulations instead of the full Landau-Fokker-Planck or the full Boltzmann operator, see for example  \cite{CROUSEILLES2004546,Gorji,Hepp}.}

\textcolor{black}{Another approximation of the Boltzmann operator is the Bathnagar-Gross-Krook (BGK) operator. It provides an approximation of the Boltzmann operator maintaining the conservation properties, the H-Theorem and the shape of equilibrium. A derivation of the BGK operator from the Boltzmann operator can be found in \cite{Pirner2018}. Here, the distribution function depending on the velocities after the interaction are replaced by the equilibrium and therefore assumes a close-to-equilibrium regime.}

Besides the one species BGK and Fokker-Planck model, there exist many extensions in the literature. \textcolor{black}{See  for example \cite{cite-key}, for BGK and Fokker-Planck models of the Boltzmann equation for gases with discrete levels of vibrational energy. In this paper, we will focus on extensions to gas mixtures. } There are many BGK models for gas mixtures proposed in the literature \cite{gross_krook1956,hamel1965,Greene,Garzo1989,Sofonea2001,Pirner,haack,Bobylev,AndriesAokiPerthame2002}, many of which satisfy the basic requirements of conservation properties and H-Theorem; and, in addition, are able to match some prescribed relaxation rates and/or transport coefficients that come from more complicated physics models or from experiment.  Many of these approaches have been extended to accommodate ellipsoid statistical (ES-BGK) models, polyatomic molecules, chemical reactions, velocity dependent collision frequencies or quantum gases; \textcolor{black}{see for example }\cite{Brull,Brull_2012,Pirner2,Todorova,Groppi,Pirner5,Pirner6,Bisi,Bisi2,Pirner9,velocity,Bisi_2018}.  Concerning the literature on multi-species Fokker Planck models, there are less results than in the BGK case, but the interest in multi-species Fokker-Planck models has been increased more and more recently. Models for gas mixtures can be found in \cite{7, Hu, Gorji, Hepp,Agrawal_Singh_Ansumali_2020}. 
The diffusion limit of a kinetic Fokker-Planck system for charged particles towards the Nernst-Planck equations was proved in \cite{15}. Furthermore, in \cite{7,11}, the limit of vanishing electron-ion mass ratios for non-homogeneous kinetic Fokker-Planck systems was investigated. In \cite{Hu}, the authors provide the first existence analysis of a multi-species Fokker-Planck system of the shape above. The works \cite{Gorji, Hepp} provide an extended Fokker-Planck model for hard-spheres gas mixtures to be able to also capture correct diffusion coefficients, mixture viscosity and heat conductivity coefficients in the hydrodynamic regime of the Navier-Stokes equations. 

The aim of this paper is to present a general multi-species Fokker-Planck model with free parameters 
to be able to fix exchange terms of momentum and energy.
Additionally, we want to study the conservation properties, positivity of all temperatures and the H-Theorem. The models \cite{7, Hu, Gorji, Hepp} can be shown to be a special case of this model presented here. This provides the possibility to create different exchange terms of momentum and energy in the macroscopic equations. 

The outline of this paper is as follows: in section \ref{secrev}, we briefly review a BGK model for gas mixtures from the literature to motivate our way to construct a Fokker-Planck model for gas mixtures in section  \ref{sec1}. We prove conservation properties of this model in section \ref{sec2}, positivity of all temperatures in section \ref{sec3} and an H-Theorem in section \ref{sec5}. In section \ref{sec4} we derive macroscopic equations and discuss several special cases in the literature \cite{7, Hu, Gorji, Hepp}. In section \ref{secN}, we consider the $N$-species case.
   
\section{Short review of a BGK model for gas mixtures from the literature}
\label{secrev}

In this section, we will briefly review an already existing BGK model from the literature to motivate how we construct the Fokker-Planck model for gas mixtures later on. For simplicity, we consider a gas mixture consisting of two species. \textcolor{black}{In the gas mixture case, one can find two types of BGK models. Just like the Boltzmann equation for gas mixtures contains a sum of collision terms on the right-hand side, one type of BGK models also have a sum of BGK-type interaction terms in the relaxation operator. The other type of models contains only one collision term on the right-hand side. In this paper, we are interested in the first type of models. } \begin{align*}
\partial_t f_1 + v \cdot \nabla_x f_1 = Q_{11}(f_1,f_1) + Q_{12}(f_1,f_2)=: Q_1^{BGK}(f_1,f_2)
\\
\partial_t f_2 + v \cdot \nabla_x f_2  = Q_{22}(f_2,f_2)+ Q_{21}(f_2,f_1)=:. Q_2^{BGK}(f_2,f_1),
\end{align*} 
Here, $f_i(x,v,t)> 0, ~i=1,2\textcolor{black}{,}$ is the distribution function of species $i$ where $x\in \mathbb{R}^d$ and $v\in \mathbb{R}^d$ are the phase space variables in dimension $d\geq 1$ and $t\geq 0$ the time. The collision operator on the right-hand side consists of a term $Q_{ii}$ describing the interactions of particles of the species \textcolor{black}{i} with  itself and a sum of collision operators $Q_{ij}, i\neq j,$ describing the interactions of particles of the  species $i$ with particles of species $j$. 
The collision operators are of the form
\begin{align} \begin{split} \label{BGK}
Q^{BGK}_1(f_1,f_2) = \textcolor{black}{ \nu_{11} n_1 (M_1-f_1) }+  \nu_{12} n_1 (M_{12} - f_1) \\
Q^{BGK}_2(f_2,f_1) = \textcolor{black}{ \nu_{22} n_2 (M_2-f_2) }+  \nu_{21} n_2 (M_{21} - f_2)
\end{split}
\end{align}
Here, $M_i$ and $M_{ij}$ denote locally Maxwell distribution functions of the form
\begin{align} 
&M_{i}(x,v,t) = \frac{n_i}{\sqrt{2 \pi \frac{T_i}{m_i}}^d} \exp({- \frac{|v-u_i|^2}{2 \frac{T_i}{m_i}}}), ~i=1,2
\\
&M_{ij}(x,v,t) = \frac{n_{ij}}{\sqrt{2 \pi \frac{T_{ij}}{m_j}}^d} \exp({- \frac{|v-u_{ij}|^2}{2 \frac{T_{ij}}{m_i}}}), ~i \neq j, ~i,j=1,2
\label{BGKmix}
\end{align}
\textcolor{black}{In equation \eqref{BGK},} $\nu_{ij} n_i$ denote the collision frequencies of species $i$ with species $j$. The parameters $\nu_{ij}$ are assumed to be positive and only depend on $x$ and $t$.  The parameters $n_i,n_{ij},u_i,u_{ij},T_i,T_{ij}$ are determined such that we have the following conservation properties:
conservation of mass, momentum and energy of the individual species in interaction with the species itself:
\begin{enumerate} 
\item $\int_{\mathbb{R}^d} Q_{ii}^{BGK}(f_i,f_i) dv = 0 \quad \text{for} \quad i = 1,2, $
\item  $\int_{\mathbb{R}^d}  m_i v Q^{BGK}_{ii}(f_i,f_i) dv= 0 \quad \text{ for} \quad i = 1,2 $
\item  $\int_{\mathbb{R}^d} m_i |v|^2 Q^{BGK}_{ii}(f_i,f_i) dv= 0 \quad \text{ for} \quad  i = 1,2. $
\end{enumerate}
Conservation of total mass, momentum and energy 
\begin{enumerate}
\item $\int_{\mathbb{R}^d} Q^{BGK}_{ij}(f_i,f_j) dv = 0 \quad \text{for} \quad i = 1,2,$  
\item $ \int_{\mathbb{R}^d}( m_1 v Q^{BGK}_{12}(f_1,f_2)+ m_2 v Q^{BGK}_{21}(f_2,f_1)) dv = 0, $
\item $ \int_{\mathbb{R}^d} (m_1 |v|^2 Q^{BGK}_{12}(f_1,f_2)+m_2 |v|^2 Q^{BGK}_{21}(f_2,f_1) )dv = 0. $
\end{enumerate}
\label{cons_prop}
for $i,j=\textcolor{black}{1,2, i\neq j}$. For this,  we relate the distribution functions to  macroscopic quantities by mean-values of $f_i$
\begin{align}
\begin{split}
\int_{\mathbb{R}^d} f_i(v) \begin{pmatrix}
1 \\ v \\ m_i |v-u_i|^2 \\ 
\end{pmatrix} 
dv =: \begin{pmatrix}
n_i \\ n_i u_i \\ d n_i T_i 
\end{pmatrix},
\end{split}
\label{macrosqu}
\end{align}
where $n_i$ is the number density, $u_i$ the mean velocity and $T_i$ the temperature which is related to the pressure $p_i$ by $p_i=n_i T_i$. Note that in this paper we shall write $T_i$ instead of $k_B T_i$, where $k_B$ is Boltzmann's constant. 
A general BGK model for \textcolor{black}{ two species} which contains most of the BGK models for gas mixtures in the literature is provided in \cite{Pirner}. We will introduce this model briefly in the following. If we assume   \begin{align} n_{12}=n_1 \quad \text{and} \quad n_{21}=n_2
\label{density} 
\end{align}
in \eqref{BGKmix}, we have conservation of the number of particles, see Theorem 2.1 in \cite{Pirner}.
If we further assume that $u_{12}$ is a linear combination of $u_1$ and $u_2$
 \begin{align}
u_{12}= \delta u_1 + (1- \delta) u_2, \quad \delta \in \mathbb{R},
\label{convexvel}
\end{align} then we have conservation of total momentum
provided that
\begin{align}
u_{21}=u_2 - \frac{m_1}{m_2} \varepsilon (1- \delta ) (u_2 - u_1),
\label{veloc}
\end{align}
see Theorem 2.2 in \cite{Pirner}. \textcolor{black}{This is again a linear combination of $u_1$ and $u_2$ with different coefficients in front of $u_1$ and $u_2$. If we see the parameter $\delta$  as a function of the masses $m_1$ and $m_2$, one can see an interpretation of this repartition. For more details, see remark 2.3 in \cite{Pirner}.}

If we further assume that $T_{12}$ is of the following form
\begin{align}
\begin{split}
T_{12} &=  \alpha T_1 + ( 1 - \alpha) T_2 + \gamma |u_1 - u_2 | ^2,  \quad 0 \leq \alpha \leq 1, \gamma \geq 0 ,
\label{contemp}
\end{split}
\end{align}
then we have conservation of total energy
provided that
\begin{align}
\begin{split}
T_{21} =\left[ \frac{1}{d} \varepsilon m_1 (1- \delta) \left( \frac{m_1}{m_2} \varepsilon ( \delta - 1) + \delta +1 \right) - \varepsilon \gamma \right] |u_1 - u_2|^2 \\+ \varepsilon ( 1 - \alpha ) T_1 + ( 1- \varepsilon ( 1 - \alpha)) T_2,
\label{temp}
\end{split}
\end{align}
see Theorem 2.3 in \cite{Pirner}.
In order to ensure the positivity of all temperatures, the parameters $\delta$ and $\gamma$ are restricted to
 \begin{align}
0 \leq \gamma  \leq \frac{m_1}{d} (1-\delta) \left[(1 + \frac{m_1}{m_2} \varepsilon ) \delta + 1 - \frac{m_1}{m_2} \varepsilon \right],
 \label{gamma}
 \end{align}
and
\begin{align}
 \frac{ \frac{m_1}{m_2}\varepsilon - 1}{1+\frac{m_1}{m_2}\varepsilon} \leq  \delta \leq 1,
\label{gammapos}
\end{align}
see Theorem 2.5 in \cite{Pirner}.

 Moreover, it can be shown that this model satisfies an H-Theorem, see Theorem 2.4 in \cite{Pirner}, meaning that we have the following inequality 
$$\sum_{\textcolor{black}{i,j=1,2}} \int_{\mathbb{R}^d} Q^{BGK}_{ij}(f_i, f_j) \log f_i dv \leq 0$$
with equality if and only if $f_i,f_j$ are Maxwell distributions with the same mean velocity and temperature. 

In the following, we will briefly motivate the meaning and possible choices of the free parameters $\alpha, \delta, \gamma$, for more details see \cite{Sandra}. One possibility is that one can choose the parameters such that one can generate special cases in the literature \textcolor{black}{\cite{gross_krook1956,hamel1965,Greene,Garzo1989,Sofonea2001,Pirner,haack,Bobylev}. }For instance if one chooses  $\varepsilon=1$, $ \delta= \frac{m_1}{m_1+m_2}$, $\alpha=\frac{m_1^2+m_2^2}{(m_1+m_2)^2}$ and $\gamma=\frac{m_1 m_2}{(m_1+m_2)^2} \frac{m_2}{d}$, one obtains the model by Hamel in \cite{hamel1965}. In \cite{Groppi} such relaxation parameters are used to fix in the continuum limit Fick`s law for diffusion velocities and Newton's law for viscous stress in the relevant set of Navier-Stokes equations.

Another possibility  is to choose the parameters in a way such that the macroscopic exchange terms of momentum and energy can be matched in a certain way for example that they coincide with the ones for the Boltzmann equation. For this, we first present the macroscopic equations with exchange terms of the BGK model \eqref{BGK}.
If we multiply the BGK model for gas mixtures by $1, m_j v, m_j \frac{|v|^2}{2}$ and integrate with respect to $v$, we obtain the following macroscopic conservation laws
\begin{multline*}
\\
\partial_t n_1 + \nabla_x \cdot (n_1 u_1)=0, 
\\
\partial_t n_2 + \nabla_x \cdot (n_2 u_2)=0,
\\
 \partial_t(m_1 n_1 u_1)+\nabla_x \cdot \int_{\mathbb{R}^d} m_1 v \otimes v f_1(v) dv + \nabla_x \cdot (m_1 n_1 u_1 \otimes u_1  )  =  f_{m_{1,2}},
\\
 \partial_t(m_2 n_2 u_2)+\nabla_x \cdot \mathbb{P}_2 + \nabla_x \cdot (m_2 n_2 u_2 \otimes u_2  )  = 
f_{m_{2,1}},
\\
\partial_t \left(\frac{m_1}{2} n_1 |u_1|^2 + \frac{3}{2} n_1 T_1 \right) + \nabla_x \cdot \int_{\mathbb{R}^d} m_1 |v|^2 v f_1(v) dv  =  F_{E_{1,2}},
\\
\partial_t \left(\frac{m_2}{2} n_2 |u_2|^2 + \frac{3}{2} n_2 T_2 \right) + \nabla_x \cdot \int_{\mathbb{R}^d} m_2 |v|^2 v f_2(v) dv  = F_{E_{2,1}},
\\
\end{multline*}
with exchange terms $f_{m_{i,j}}$ and $F_{E_{i,j}}$ given by
\begin{align*}
f_{m_{1,2}}&= - f_{m_{2,1}} = m_1 \nu_{12} n_1 n_2 (1 - \delta) (u_2 - u_1), \\
F_{m_{1,2}}&= - F_{m_{2,1}} \\& = \left[\nu_{12} \frac{1}{2} n_1 n_2 m_1 (\delta -1) (u_1 + u_2 + \delta(u_1-u_2)) + \frac{1}{2} \nu_{12} n_1 n_2 \gamma (u_1 - u_2) \right] \cdot (u_1-u_2)\\ &+ \frac{d}{2} \varepsilon \nu_{21} n_1 n_2 (1-\alpha) (T_2-T_1).
\end{align*}

Here, one can observe a physical meaning of $\alpha$ and $\delta$. We see that $\alpha$ and $\delta$ show up in the exchange terms of momentum and energy as parameters in front of the relaxation of $u_1$ towards $u_2$ and $T_1$ towards $T_2$. So they determine, together with the collision frequencies, the speed of relaxation of the mean velocities and the temperatures to a common value. 

Here now, as it is done in section 4.1 in \cite{haack} or section 4 in \cite{Sandra}, \textcolor{black}{one can} compare the relaxation rates in the space-homogeneous case to the relaxation rates for the space-homogeneous Boltzmann equation. In \cite{haack}, they find values for $\nu_{kj}$ such that either the relaxation rate for the mean velocities or the relaxation for the temperatures coincides with the corresponding rate of the Boltzmann equation. But using the free parameters $\alpha$, $\delta$ and $\gamma$ one is able to match both of the relaxation rates at the same time. 

 \section{General multi-species Fokker-Planck model}
\label{sec1}
In this section, we present a general multi-species Fokker-Planck model and study the conservation properties, an entropy inequality, the expected shape in equilibrium (Maxwell distribution with common mean velocity and temperature) and the positivity of all temperatures. For simplicity, we present this model for two-species, but everything can be extended to a general number of $N$ species, since we made the assumption of only considering binary interactions, see section \ref{secN}. So in this section, we consider the following system of Fokker-Planck equations
{\small
\begin{align}
\partial_t f_1 + v \cdot \nabla_x f_1 =  c_{11}n_1 \text{ div} (\nabla_v (\frac{T_1}{m_1} f_1) + (v-u_{1}) f_1 ) + c_{12}n_2 \text{div} (\nabla_v (\frac{T_{12}}{m_1} f_1) + (v-u_{12}) f_1 ) \\
\partial_t f_1 + v \cdot \nabla_x f_1 =  c_{22} n_2 \text{ div} (\nabla_v (\frac{T_2}{m_2} f_2) + (v-u_{2}) f_2 ) + c_{21}  n_1\text{div} (\nabla_v (\frac{T_{21}}{m_2} f_2 )+  (v-u_{21}) f_2 )
 \label{FP}\end{align}}
The quantity $c_{ij}$ is a friction constant, see \cite{Toscani} for a motivation of the one species case.
To be flexible in choosing the relationship between the constants $c_{12}, c_{21}$, we now assume the relationship 
\begin{align}
c_{12}  = \varepsilon ~ c_{21}, \varepsilon \leq 1 \text{ and } \varepsilon \frac{m_1}{m_2} \leq 1
\label{coll}
\end{align}
Note, that the assumption on $\varepsilon$ covers the two common cases in the literature for $\varepsilon$ which are $\varepsilon=\frac{m_2}{m_1}$ and $\varepsilon =1$ if the notation of $1$ and $2$ is chosen in a suitable way.
\subsection{Conservation properties}
\label{sec2}
This section shows how the macroscopic quantities $u_{ij}, T_{ij}$ in the interspecies Maxwell distributions have to be chosen in order to ensure the macroscopic conservation properties. We note that the mass is automatically conserved. 

\textcolor{black}{We start with the operator $Q^{FP}_{kk}, ~ k=1,2$ which describes the interactions of a species with itself. This operator satisfies the following conservation properties. For the proof see the one species case \cite{Gorji}.
\begin{theorem}[Conservation properties of the operator for intra species interactions]
\label{constraints_intra}
We have conservation of mass, momentum and energy 
$$
\int_{\mathbb{R}^d} m_k \begin{pmatrix}1, v, |v|^2 \end{pmatrix}^T Q^{FP}_{kk}(f_k,f_k)dv  = 0, ~k=1,2
$$
\end{theorem}
In the case of the mixture interaction terms, we can prove
\begin{theorem}[Conservation of total momentum]
Assume the condition \eqref{coll} for the collision frequencies and that $u_{12}$ is a linear combination of $u_1$ and $u_2$
 \begin{align}
u_{12}= \delta u_1 +(1- \delta) u_2, \quad \delta \in \mathbb{R}.
\label{convexvel}
\end{align} Then we have conservation of total momentum
$$
\int_{\mathbb{R}^d} m_1 v Q^{FP}_{12}(f_1,f_2)dv +
\int_{\mathbb{R}^d} m_2 v Q^{FP}_{21}(f_2,f_1) dv = 0, 
$$
provided that
\begin{align}
u_{21}=u_2 -(1- \delta )\varepsilon \frac{m_1}{m_2} (u_2 - u_1).
\label{veloc}
\end{align}
\label{theomom}
\end{theorem}}
\begin{proof}
The flux of momentum of species $1$ is given by
\begin{align}
\begin{split}
f_{m_{1,2}}:&=  c_{12} m_1 n_2 \int_{\mathbb{R}^d} v \text{ div} (\nabla_v (\frac{T_{12}}{m_1} f_1) + (v-u_{12}) f_1 ) dv\\&= -m_1 c_{12} n_2 \int_{\mathbb{R}^d} (\nabla_v (\frac{T_{12}}{m_1} f_1) + (v-u_{12}) f_1) dv \\&= m_1 c_{12} n_1 n_2 (u_{12} - u_1)
=
m_1 c_{12} n_1 n_2 (1 - \delta) (u_2 - u_1). 
\end{split}
\label{flux_mom_12}
\end{align}
The flux of momentum of species $2$ is given by
\begin{align}
f_{m_{2,1}}
=m_2 c_{21} n_2 n_1 (u_{21} - u_2).
\label{flu_mom_21}
\end{align}
In order to get conservation of momentum we therefore need
\begin{align*}
m_1 c_{12} n_1 n_2 (1 - \delta) (u_2 - u_1)+ m_2 c_{21} n_1 n_2 (u_{21} - u_2)= 0 ,
\end{align*}
which holds provided $u_{21}$ satisfies
 \eqref{veloc} 
\end{proof}
\begin{remark}
If we write $\tilde{\delta}= 1-\frac{m_1}{m_2}\varepsilon (1-\delta)$ we obtain a similar structure for $u_{21}$ as for $u_{12}$
$$ u_{21} = \tilde{\delta} u_2 + (1- \tilde{\delta}) u_1.$$
\end{remark}
\begin{theorem}[Conservation of total energy]
Assume conditions \eqref{convexvel} and \eqref{veloc} and assume that $T_{12}$ is of the following form
\begin{align}
\begin{split}
T_{12}=  \alpha T_1 + ( 1 - \alpha) T_2 + \gamma |u_1 - u_2 | ^2,  \quad 0 \leq \alpha \leq 1, \gamma \geq 0.
\label{contemp}
\end{split}
\end{align}
Then we have conservation of total energy
$$
\int_{\mathbb{R}^d} \frac{m_1}{2} |v|^2 Q^{FP}_{12}(f_1,f_2)dv +
\int_{\mathbb{R}^d} \frac{m_2}{2} |v|^2 Q^{FP}_{21}(f_2,f_1)dv = 0,
$$
\end{theorem}
provided that
\begin{align}
\begin{split}
T_{21}=\left[ \frac{1}{d} \varepsilon m_1 (1- \delta)  - \varepsilon \gamma \right] |u_1 - u_2|^2 + \varepsilon ( 1 - \alpha ) T_1 + ( 1- \varepsilon ( 1 - \alpha)) T_2.
\label{temp}
\end{split}
\end{align}
\label{theoen}
\begin{proof}
Using the energy flux of species $1$ 
\begin{align*}
F_{E_{1,2}}&:=c_{11} n_1 \int_{\mathbb{R}^d} \frac{m_1}{2} |v|^2 \text{ div} (\nabla_v(\frac{T_1}{m_1} f_1) + (v-u_1)) f_1  dv 
\\ &+ c_{12} n_2 \int_{\mathbb{R}^d} \frac{m_1}{2} |v|^2  \text { div} (\nabla_v(\frac{T_{12}}{m_1} f_1) + (v-u_{12})) f_1 dv 
\\ &=- c_{12} n_2 m_1 \int_{\mathbb{R}^d} v \cdot (\nabla_v(\frac{T_{21}}{m_1} f_1) + (v- u_{12})) f_1 dv 
 \\ &= c_{12} n_2 m_1d \int_{\mathbb{R}^d}  \frac{T_{12}}{m_1} f_1 dv - c_{12} n_2 m_1 \int_{\mathbb{R}^d} v \cdot (v-u_{12}) f_1 dv \\ &= c_{12} n_2 d n_1 T_{12} - c_{12} n_2 m_1 (d n_1 \frac{T_1}{m_1} + n_1 |u_1|^2 - n_1 u_1 \cdot u_{12}) \\ &= c_{12} n_2 d n_1 T_{12} - c_{12} n_2  n_1 (d T_1 + m_1 u_1 \cdot (u_1-u_{12})) \\ &= c_{12} n_1 n_2  d (T_{12}-T_1) - c_{12} m_1 n_1 n_2 u_1 \cdot (u_1-u_{12}) \\ &= c_{12} n_1 n_2 (1 -\alpha) (T_2-T_1) - c_{12} m_1 n_2 n_1( (1- \delta) u_1 \cdot (u_1-u_2) + \gamma |u_1-u_2|^2) 
 \end{align*}
where we used \eqref{convexvel} and \eqref{contemp}. Analogously the energy flux of species $2$ towards $1$ is 
\begin{align*}
F_{E_{2,1}}= c_{21} m_2 n_1 n_2 ( u_2 \cdot (u_{21}-u_2))+ d~ c_{21} n_1 n_2 ( T_{21} -T_2) \\ = c_{21} m_2 n_1 n_2 (1-\delta) \frac{m_1}{m_2} \varepsilon (u_2 \cdot (u_1-u_2))+ d ~c_{21} n_1 n_2 ( T_{21} -T_2)
\end{align*}
Here, we substituted $u_{21}$ with \eqref{veloc}. 
Adding these two terms, we see that the total energy is conserved  provided that $T_{21}$ is given by \eqref{temp}. \end{proof}
\begin{remark}
We have $0 \leq 1-\varepsilon (1 - \alpha) \leq 1$ and  $0 \leq \varepsilon (1- \alpha) \leq 1$, so that in \eqref{temp} the two terms with the temperatures are also a convex combination of $T_1$ and $T_2$.  
\end{remark}
\subsection{Positivity of the temperatures}
\label{sec3}
\begin{theorem}
Assume that $f_1(x,v,t), f_2(x,v,t) > 0$. Then all temperatures $T_1$, $T_2$, $T_{12}$ given by \eqref{contemp} and $T_{21}$ given by \eqref{temp} are positive provided that 
 \begin{align}
0 \leq \gamma \leq \frac{m_1}{d} (1-\delta) 
 \label{gamma}
 \end{align}
\end{theorem}
\begin{proof}
$T_1$ and $T_2$ are positive as integrals of positive functions. $T_{12}$ is positive because by construction it is a convex combination of $T_1$ and $T_2$. For $T_{21}$ we consider the coefficients in front of $|u_1-u_2|^2$, $T_1$ and $T_2$. The term in front of $T_1$ is positive by definition. The positivity of the term in front of $T_2$ is equivalent to the condition $\alpha \geq 1- \frac{1}{\varepsilon}$, which is satisfied since $\varepsilon \leq1$, the positivity of the term in front of $|u_1-u_2|^2$ is equivalent to the condition \eqref{gamma}.
\end{proof}
\begin{remark}
According to the definition of $\gamma$, $\gamma$ is a non-negative number, so the right-hand side of the inequality in \eqref{gamma} must be non-negative. This condition is equivalent to 
\begin{align}
 \delta \leq 1 .
\label{gammapos}
\end{align}
\end{remark}
\subsection{Macroscopic equations and exchange terms of momentum and energy}
\label{sec4}
In this section, we deal with macroscopic equations, exchange terms of momentum and energy, and special cases in the literature.  With a specific choice of the parameters  we can generate special cases in the literature \cite{7, Hu, Gorji}. For instance, in \cite{7} the mean mixture velocities and temperatures are chosen to be
\begin{align*}
u_{12}= u_{21} = \frac{u_1+u_2}{2}; \quad T_{12}=T_{21} = \frac{m_2 T_1 +m_1 T_2}{m_1+m_2} + \frac{m_1 m_2}{m_1+m_2} \frac{1}{2d} |u_1-u_2|^2
\end{align*}
so we can generate this model by choosing 
\begin{align*}
\alpha= \frac{m_2}{m_1+m_2}, \quad \delta=\frac{1}{2}, \quad  \gamma= \frac{1}{2d} \frac{m_1 m_2}{m_1+m_2}
\end{align*}
With the choice of 
\begin{align*}
\alpha= \frac{c_{12} n_2}{c_{12} n_2 + c_{21} n_1},  \delta = \frac{c_{12} m_1 n_1}{c_{12} m_1 n_1 + c_{12} m_2 n_2}, \gamma = \frac{c_{12} m_1 n_1 c_{21} m_2 n_2}{d (c_{12} n_1 + c_{21} n_2)(c_{21} n_1 n_1 +c_{12} m_2 n_2)}
\end{align*}
we can generate $u_{12}, u_{21}, T_{12}, T_{21}$ as in \cite{Hu} given by
\begin{align*}
u_{12}&=u_{21}= \frac{c_{21} m_1 n_1 u_1 + c_{12} m_2 n_2 u_2}{c_{12} m_2 n_2 +c_{21} m_1 n_1}, \\ T_{12}&=T_{21}= \frac{c_{21} n_1 T_1 + c_{12}n_2 T_2}{c_{12} n_2 + c_{21} n_1} + \frac{c_{12} m_1 n_1 c_{21} m_2 n_2}{d (c_{12} n_1 + c_{21} n_2)(c_{21} n_1 n_1 +c_{12} m_2 n_2)} |u_1-u_2|^2
\end{align*}
Another possibility for example is to choose as mixture velocity the velocity of the other species as it is done for example in \cite{Gorji} with $\delta=0$ to have 
\begin{align*} 
u_{12}=u_2, u_{21}=u_1.
\end{align*}

Another way to see the influence of the parameters is in the macroscopic exchange terms of momentum and energy. For this, we first present the macroscopic equations with exchange terms of the Fokker-Planck model \eqref{FP}.
If we multiply the Fokker-Planck model for gas mixtures by $1, m_j v, m_j \frac{|v|^2}{2}$ and integrate with respect to $v$, we obtain the following macroscopic conservation laws
\begin{multline*}
\\
\partial_t n_1 + \nabla_x \cdot (n_1 u_1)=0, 
\\
\partial_t n_2 + \nabla_x \cdot (n_2 u_2)=0,
\\
 \partial_t(m_1 n_1 u_1)+\nabla_x \cdot \int_{\mathbb{R}^d} m_1 v \otimes v f_1(v) dv   =  f_{m_{1,2}},
\\
 \partial_t(m_2 n_2 u_2)+\nabla_x \cdot  \int_{\mathbb{R}^d} m_2 v \otimes v f_2(v) dv  = 
f_{m_{2,1}},
\\
\partial_t \left(\frac{m_1}{2} n_1 |u_1|^2 + \frac{3}{2} n_1 T_1 \right) + \nabla_x \cdot \int_{\mathbb{R}^d} m_1 |v|^2 v f_1(v) dv  =  F_{E_{1,2}},
\\
\partial_t \left(\frac{m_2}{2} n_2 |u_2|^2 + \frac{3}{2} n_2 T_2 \right) + \nabla_x \cdot \int_{\mathbb{R}^d} m_2 |v|^2 v f_2(v) dv  = F_{E_{2,1}},
\\
\end{multline*}
with exchange terms $f_{m_{i,j}}$ and $F_{E_{i,j}}$ given by
\begin{align}
\begin{split}
f_{m_{1,2}}&= - f_{m_{2,1}} = m_1 c_{12} n_1 n_2 (1 - \delta) (u_2 - u_1), \\
F_{E_{1,2}}&= - F_{E_{2,1}} \\& = c_{12}  n_1 n_2 m_1 (1-\delta) u_1  \cdot (u_2-u_1)  +  \gamma \frac{d}{m_1} |u_1-u_2|^2 \\ &+ d  c_{12} n_1 n_2 (1-\alpha) (T_2-T_1).\label{macrosequ1}
\end{split}
\end{align}

\subsection{H-theorem for mixtures}
\label{sec5}
In this section we will prove an H-Theorem for the model \eqref{FP}. For this, we make the following additional assumptions on the free parameters. We make the stronger assumptions of \eqref{contemp}, \eqref{gamma}, \eqref{gammapos}. 
\begin{align}
 \frac{\varepsilon}{1+\varepsilon} \leq \alpha \leq 1
, \quad 
\frac{\varepsilon}{1+\varepsilon} \leq \delta \leq 1
, \quad 
(1-\delta)^2 \frac{m_1}{d} \leq \gamma \leq (1-\delta) \frac{m_1}{d} \frac{\varepsilon}{1+\varepsilon}
\label{asspar}
\end{align}
Moreover, in order to simplify the notation we define the following quantities 
\begin{align}
\gamma_1:= (1-\delta)^2 \frac{m_1}{d}; \quad \gamma_2:= (1-\delta)^2 \frac{m_2}{d} \varepsilon^2 (\frac{m_1}{m_2})^2, \quad \tilde{\gamma}:= \frac{m_1}{d} \varepsilon (1-\delta) - \varepsilon \gamma
\label{defgam}
\end{align}
and the temperatures
\begin{align}
\bar{T}_{12} = \alpha T_1+ (1- \alpha) T_2; \quad \bar{T}_{21}= \varepsilon (1-\alpha) T_1 + (1-\varepsilon (1-\alpha)) T_2
\label{Tbar0}
\end{align}
We start with some lemmas which we will need later for the proof of the H-Theorem.
\begin{lemma}
Let $M_{12}, M_{21}$ \textcolor{black}{be} the two Maxwell distributions given by \eqref{BGKmix}. Then we have
\begin{align*}
&\frac{T_{12}}{m_1} c_{12} n_2 \int_{\mathbb{R}^d} \frac{M_{12}^2}{f_1} \left(\frac{\nabla_v f_1 M_{12}- \nabla_v M_{12} f_1}{M_{12}^2} \right)^2 dv \\&+ \frac{T_{21}}{m_2} c_{21} n_1 \int_{\mathbb{R}^d} \frac{M_{21}^2}{f_2} \left(\frac{\nabla_v f_2 M_{21}- \nabla_v M_{21} f_2}{M_{21}^2} \right)^2 dv \\ &= \frac{T_{12}}{m_1} c_{12} n_2 \int_{\mathbb{R}^d} \frac{|\nabla_v f_1|^2}{f_1} dv + \frac{T_{21}}{m_2} c_{21} n_1 \int_{\mathbb{R}^d} \frac{|\nabla_v f_2|^2}{f_2} dv \\&+ c_{12} n_2 n_1 d \frac{T_1 + \frac{m_1}{d} (1-\delta)^2 |u_1-u_2|^2}{T_{12}} +  c_{21} n_2 n_1 d \frac{T_2 + \frac{m_2}{d}\varepsilon^2 (\frac{m_1}{m_2})^2 (1-\delta)^2 |u_1-u_2|^2}{T_{21}}\\& - 2(1+ \varepsilon) c_{21} n_1 n_2 d
\end{align*}
\label{inequ}
\end{lemma}
\begin{proof}
We can compute 
{\small
\begin{align*}
&\frac{T_{12}}{m_1} c_{12} n_2 \int_{\mathbb{R}^d} \frac{M_{12}^2}{f_1} \left(\frac{\nabla_v f_1 M_{12}- \nabla_v M_{12} f_1}{M_{12}^2} \right)^2 dv + \frac{T_{21}}{m_2} c_{21} n_1 \int_{\mathbb{R}^d} \frac{M_{21}^2}{f_2} \left(\frac{\nabla_v f_2 M_{21}- \nabla_v M_{21} f_2}{M_{21}^2} \right)^2 dv \\ &=  \frac{T_{12}}{m_1} c_{12} n_2 \int_{\mathbb{R}^d} \frac{|\nabla_v f_1|^2}{f_1} dv + \frac{T_{21}}{m_2} c_{21} n_1 \int_{\mathbb{R}^d} \frac{|\nabla_v f_2|^2}{f_2} dv \\&+ c_{12} n_2  \int_{\mathbb{R}^d} \frac{|v-u_{12}|^2}{T_{12}/m_1} f_1 dv  + c_{21}  n_1 \int_{\mathbb{R}^d} \frac{|v-u_{21}|^2}{T_{21}/m_1} f_2 dv \\ &+  2 c_{12} n_2 \int_{\mathbb{R}^d} \nabla_v f_1 \cdot (v-u_{12})dv + 2 c_{21} n_1 \int_{\mathbb{R}^d} \nabla_v f_2 \cdot (v-u_{21})dv 
\\ &= \frac{T_{12}}{m_1} c_{12} n_2 \int_{\mathbb{R}^d} \frac{|\nabla_v f_1|^2}{f_1} dv + \frac{T_{21}}{m_2} c_{21} n_1 \int_{\mathbb{R}^d} \frac{|\nabla_v f_2|^2}{f_2} dv \\&+ c_{12} n_2 n_1 d \frac{T_1 + \frac{m_1}{d}  |u_1-u_{12}|^2}{T_{12}}  + c_{21} n_2 n_1 d \frac{T_2 + \frac{m_2}{d} \varepsilon^2 (\frac{m_1}{m_2})^2 (1-\delta)^2 |u_2-u_{21}|^2}{T_{21}} \\ &- 2(1+ \varepsilon) c_{21} n_1 n_2 d
\\ &= \frac{T_{12}}{m_1} c_{12} n_2 \int_{\mathbb{R}^d} \frac{|\nabla_v f_1|^2}{f_1} dv + \frac{T_{21}}{m_2} c_{21} n_1 \int_{\mathbb{R}^d} \frac{|\nabla_v f_2|^2}{f_2} dv \\&+ c_{12} n_2 n_1 d \frac{T_1 + \frac{m_1}{d} (1-\delta)^2 |u_1-u_2|^2}{T_{12}} + + c_{21} n_2 n_1 d \frac{T_2 + \frac{m_2}{d} \varepsilon^2 (\frac{m_1}{m_2})^2 (1-\delta)^2 |u_1-u_2|^2}{T_{21}} \\&- 2(1+ \varepsilon) c_{21} n_1 n_2 d
\end{align*}}
Here, we used $\nabla_v M_{12} = - \frac{v-u_{12}}{T_{12}/m_1} M_{12}$ and the relationship \eqref{convexvel} and \eqref{veloc} for $u_{12}, u_{21}$.
\end{proof}
\begin{lemma}
We assume the estimate for $\alpha$ in \eqref{asspar}. Then we have
\begin{align}
\varepsilon T_1 \bar{T}_{21} + \bar{T}_{12} T_2 \leq (1+\varepsilon) \bar{T}_{12} \bar{T}_{21} 
\label{Tbar}
\end{align}
\label{inequ2}
\end{lemma}
\begin{proof}
If we insert the expressions for $\bar{T}_{12}, \bar{T}_{21}$ given by \eqref{Tbar0} we get that \eqref{Tbar} is equivalent to
\begin{align*}
(1- \alpha) \varepsilon (\alpha - (1- \alpha) \varepsilon) (T_1-T_2)^2 \geq 0
\end{align*}
This is true if $1 \geq \alpha \geq \frac{\varepsilon}{1+ \varepsilon}$ which we assumed in \eqref{asspar}.
\end{proof}
\begin{lemma}
We assume \eqref{asspar}. Then we have
\begin{align*}
\varepsilon \gamma_1 \tilde{\gamma} + \gamma \gamma_2 \leq (1+ \varepsilon) \gamma \tilde{\gamma}
\end{align*}
\label{inequ3}
\end{lemma}
\begin{proof}
If we insert the expressions for $\gamma_1, \gamma_2, \tilde{\gamma}$ given by \eqref{defgam}, we obtain
\begin{align*}
- ((1+\varepsilon) \gamma (- \varepsilon \gamma + (1-\delta) \varepsilon \frac{m_1}{d})) + (1-\delta)^2 \varepsilon \frac{m_1}{d} (- \varepsilon \gamma + (1-\delta) \varepsilon \frac{m_1}{d}) \\+ (1-\delta)^2 \gamma \varepsilon^2 \frac{m_1}{d} \frac{m_1}{m_2} \leq 0
\end{align*}
This inequality is true if we can prove separately 
\begin{align}
\begin{split}
\varepsilon \gamma (- \varepsilon \gamma + (1-\delta) \varepsilon \frac{m_1}{d}) \geq (1-\delta)^2 \varepsilon \frac{m_1}{d} (- \varepsilon \gamma + (1-\delta) \varepsilon \frac{m_1}{d}) \\
\varepsilon \gamma^2 - (1-\delta) \varepsilon \frac{m_1}{d} \gamma + (1-\delta)^2 \varepsilon^2 \gamma \frac{m_1}{d} \frac{m_1}{m_2} \leq 0
\end{split}
\label{equin}
\end{align}
We start with the first inequality. The factor $- \varepsilon \gamma + (1-\delta) \varepsilon \frac{m_1}{d}$ is non-negative, since this is the condition ensuring positivity of the temperatures \eqref{gammapos}. Therefore, we get that $\gamma$ has to satisfy
\begin{align*}
\gamma \geq (1-\delta)^2 \frac{m_1}{d}
\end{align*}
as assumed in \eqref{asspar}. This is possible and no restriction to the upper bound ensuring the positivity \eqref{gammapos}, since we assumed $0 \leq \delta \leq 1$.  \\ Now, for the second inequality in \eqref{equin}, we divide by $\varepsilon \gamma$ to get
\begin{align*}
\gamma - (1- \delta) \frac{m_1}{d} + (1-\delta)^2 \varepsilon \frac{m_1}{d} \frac{m_1}{m_2} \leq 0
\end{align*}
which is satisfied if 
\begin{align*}
\gamma \geq (1-\delta) \frac{m_1}{d}((1-\delta) \varepsilon \frac{m_1}{m_2}-1)
\end{align*}
This is satisfied due to the estimate on $\gamma$ in \eqref{asspar} and assumption \eqref{coll}.
\end{proof}
\begin{lemma}
We assume \eqref{asspar}. Then, we have
\begin{align*}
\varepsilon T_1 \tilde{\gamma} |u_1-u_2|^2 + \varepsilon \gamma_1 |u_1-u_2|^2 \bar{T}_{21} + \bar{T}_{12} \gamma_2 |u_1-u_2|^2 + \gamma |u_1-u_2|^2 T_2  \\ \leq (1+\varepsilon) \bar{T}_{12} \tilde{\gamma} |u_1-u_2|^2 + (1+\varepsilon)  \gamma |u_1-u_2|^2 \bar{T}_{21}
\end{align*}
\label{inequ4}
\end{lemma}
\begin{proof}
We insert the expressions for $\bar{T}_{12}, \bar{T}_{21}$ given by \eqref{Tbar0} and get
\begin{align*}
&\varepsilon T_1 \tilde{\gamma} |u_1-u_2|^2 + \varepsilon \gamma_1 |u_1-u_2|^2 (\varepsilon (1-\alpha) T_1 + (1-\varepsilon (1-\alpha)) T_2) \\&+ (\alpha T_1+(1-\alpha) T_2) \gamma_2 |u_1-u_2|^2 + \gamma |u_1-u_2|^2 T_2  \\ &\leq (1+\varepsilon) (\alpha T_1+(1-\alpha) T_2) \tilde{\gamma} |u_1-u_2|^2 \\&+ (1+\varepsilon)  \gamma |u_1-u_2|^2 (\varepsilon (1-\alpha) T_1 + (1-\varepsilon (1-\alpha)) T_2)\end{align*}
We compare the coefficients in front of $T_1$ and $T_2$ and obtain the inequalities 
\begin{align*}
\varepsilon \tilde{\gamma} + \varepsilon \gamma_1 \varepsilon (1-\alpha) + \alpha \gamma_2 &\leq (1+\varepsilon) \alpha \tilde{\gamma} + (1+\varepsilon) \gamma \varepsilon (1-\alpha) \\
\varepsilon \gamma_1 (1- \varepsilon (1-\alpha)) + (1- \alpha) \gamma_2 + \gamma &\leq (1+\varepsilon)(1- \alpha) \tilde{\gamma} + (1+\varepsilon) (1- \varepsilon(1-\alpha)) \gamma
\end{align*}
We start with the first inequality. According to the definition of $\gamma_1, \gamma_2$ given by \eqref{defgam} and the lower bound on $\gamma$ given by \eqref{gammapos} and assumption \eqref{coll}, we have
\begin{align*}
\gamma_1 \leq \gamma \text{ and } \gamma_2 \leq \gamma
\end{align*}
Additionally, we observe that 
\begin{align*}
\tilde{\gamma} = \varepsilon (1-\delta) \frac{m_1}{d} - \varepsilon \gamma \geq \gamma
\end{align*}
since we assumed the stricter upper bound on $\gamma$ in \eqref{asspar}. The stricter upper bound on $\gamma$ is not a contradiction to the lower bound since we assumed $\delta \geq \frac{\varepsilon}{1+\varepsilon} in \eqref{asspar}.$
All in all, this leads to 
\begin{align*}
\varepsilon \gamma_1 \varepsilon (1-\alpha) + \alpha \gamma_2 \leq (\varepsilon^2 (1-\alpha) + \alpha) \gamma = (\varepsilon^2 (1-\alpha) + \varepsilon (1-\alpha))\gamma + (\alpha - \varepsilon (1-\alpha)) \gamma \\ \leq (\varepsilon^2 (1-\alpha) + \varepsilon (1-\alpha))\gamma + (\alpha - \varepsilon (1-\alpha)) \tilde{\gamma}
\end{align*}
which corresponds to the first inequality. The last inequality is possible since we assumed $\alpha \geq \frac{\varepsilon}{1+\varepsilon}$ in \eqref{asspar} and therefore the coefficient $\alpha- \varepsilon (1-\alpha)$ is non-negative. 
In a similar way, one can prove the second inequality.
\end{proof}

\begin{theorem}[H-theorem for \textcolor{black}{mixtures}]
Assume $f_1, f_2 >0$. 
Assume the relationship between the collision frequencies \eqref{coll}, the conditions for the interspecies Maxwellians \eqref{density}, \eqref{convexvel}, \eqref{veloc}, \eqref{contemp} and \eqref{temp} with $\alpha, \delta \neq 1$, the positivity of the temperatures \eqref{gamma} and the assumptions on the parameters \eqref{asspar}. \textcolor{black}{Then}
\begin{align*}
\int_{\mathbb{R}^d} (\ln f_1) ~ Q^{FP}_{11}(f_1,f_1) &+ (\ln f_1) ~ Q^{FP}_{12}(f_1,f_2) dv \\&+ \int_{\mathbb{R}^d} (\ln f_2) ~ Q^{FP}_{22}(f_2, f_2)+ (\ln f_2) ~ Q^{FP}_{21}(f_2, f_1) dv\leq 0 ,
\end{align*}
with equality if and only if $f_1$ and $f_2$ are Maxwell distributions with equal velocity and temperature. 
\label{H-theorem}
\end{theorem}
\begin{proof}
The fact that $\int_{\mathbb{R}^d} \ln f_k Q(f_k,f_k)dv \leq 0$, $k=1,2$ is shown in proofs of the H-theorem of the single Fokker-Planck-model, for example in \cite{Singh}.
In both cases we have equality if and only if $f_1=M_1$ and $f_2 = M_2$. \\

Let us define
\begin{align*}
I:&= \int_{\mathbb{R}^d} Q^{FP}_{12}(f_1,f_2) \ln f_1 dv + \int_{\mathbb{R}^d} Q^{FP}_{21}(f_2, f_1) \ln f_2 dv \\ &=\int_{\mathbb{R}^d} c_{12} n_1 \text{ div} ( \nabla_v ( \frac{T_{12}}{m_1} f_1) + (v-u_{12}) f_1 ) \ln f_1 dv
\\&+ \int_{\mathbb{R}^d} c_{21} n_2 \text{div} ( \nabla_v ( \frac{T_{21}}{m_2} f_2) + (v-u_{21}) f_2 ) \ln f_2 dv
\end{align*}
Integration by parts leads to
\begin{align*}
I& =- \int_{\mathbb{R}^d} c_{12} n_1 ( \nabla_v ( \frac{T_{12}}{m_1} f_1)  + (v-u_{12}) f_1 ) \frac{\nabla_v f_1}{f_1} dv
\\ &-  \int_{\mathbb{R}^d} c_{21} n_2  ( \nabla_v ( \frac{T_{21}}{m_2} f_2) + (v-u_{21}) f_2 ) \frac{\nabla_v f_2}{f_2} dv
\\ &=\int_{\mathbb{R}^d} c_{12} n_1 \frac{T_{12}}{m_1} f_1 |\frac{\nabla_v f_1}{f_1}|^2 - \int_{\mathbb{R}^d} c_{21} n_2 \frac{T_{21}}{m_2} f_2 |\frac{\nabla_v f_2}{f_2}|^2 dv \\& - \int_{\mathbb{R}^d} c_{12} n_2 (v-u_{12}) \cdot \nabla_v f_1 dv - \int_{\mathbb{R}^d} c_{21} n_1 (v-u_{21}) \cdot \nabla_v f_2 dv
\\ &=- \int_{\mathbb{R}^d} c_{12} n_2  \frac{T_{12}}{m_1} f_1 |\frac{\nabla_v f_1}{f_1}|^2 dv
-  \int_{\mathbb{R}^d} c_{21} n_2   \frac{T_{21}}{m_2} f_2 |\frac{\nabla_v f_2}{f_2}|^2  dv + c_{12} n_2 n_1 d + c_{21} n_1 n_2 d
\end{align*}
By using the relationship \eqref{coll}, we obtain
\begin{align*}
I: =- \int_{\mathbb{R}^d} c_{12} n_2  \frac{T_{12}}{m_1} f_1| \frac{\nabla_v f_1}{f_1}|^2) dv
- \int_{\mathbb{R}^d} c_{21} n_2   \frac{T_{21}}{m_2} f_2 |\frac{\nabla_v f_2}{f_2}|^2)  dv + c_{21} n_2 n_1 d (1+\varepsilon)
\end{align*}
By using lemma \ref{inequ}, we can write 
\begin{align}
\begin{split}
I &=- \frac{T_{12}}{m_1} c_{12} n_2 \int_{\mathbb{R}^d} \frac{M_{12}^2}{f_1} \left(\frac{\nabla_v f_1 M_{12}- \nabla_v M_{12} f_1}{M_{12}^2} \right)^2 dv \\& - \frac{T_{21}}{m_2} c_{21} n_1 \int_{\mathbb{R}^d} \frac{M_{21}^2}{f_2} \left(\frac{\nabla_v f_2 M_{21}- \nabla_v M_{21} f_2}{M_{21}^2} \right)^2 dv \\&+ c_{12} n_2 n_1 d \frac{T_1 + \frac{m_1}{d} (1-\delta)^2 |u_1-u_2|^2}{T_{12}}  + c_{21} n_2 n_1 d \frac{T_2 + \frac{m_2}{d} \varepsilon^2 (\frac{m_1}{m_2})^2 (1-\delta)^2 |u_1-u_2|^2}{T_{21}} \\ &- (1+ \varepsilon) c_{21} n_1 n_2 d
\label{I}
\end{split}
\end{align}
The first two terms are non-positive, so we get the claimed inequality if we can prove
\begin{align*}
&c_{12} n_2 n_1 d \left(T_1 + \frac{m_1}{d} (1- \delta)^2 |u_1-u_2|^2\right) T_{21} \\&+ c_{21} n_1 n_2 d T_{12} \left(T_2 + \frac{m_2}{d} (1-\delta)^2 \varepsilon^2 (\frac{m_1}{m_2})^2 |u_1-u_2|^2\right) \\ &\leq (1+\varepsilon) c_{21} n_1 n_2 d T_{12} T_{21}
\end{align*}
which is by using relationship \eqref{coll} equivalent to
\begin{align*}
\varepsilon  \left(T_1 + \frac{m_1}{d} (1- \delta)^2 |u_1-u_2|^2 \right) T_{21} + T_{12} \left(T_2 + \frac{m_2}{d} (1-\delta)^2 \varepsilon^2 (\frac{m_1}{m_2})^2 |u_1-u_2|^2\right) \\ \leq (1+\varepsilon)  T_{12} T_{21}\end{align*}
With the notation introduced in \eqref{defgam} and \eqref{Tbar}, we can write
\begin{align*}
\gamma_1 = \frac{m_1}{d} (1- \delta)^2, \quad \gamma_2 = \frac{m_2}{d} (1-\delta)^2 \varepsilon^2 (\frac{m_1}{m_2})^2
\end{align*}
and \textcolor{black}{therefore we have}
\begin{align*}
T_{12} =: \bar{T}_{12} + \gamma |u_1-u_2|^2, \quad T_{21} =: \bar{T}_{21} +\tilde{\gamma} |u_1-u_2|^2
\end{align*}
Then we get
\begin{align*}
\varepsilon (T_1 + \gamma_1 |u_1-u_2|^2)( \bar{T}_{21} + \tilde{\gamma} |u_1-u_2|^2) + (\bar{T}_{12} + \gamma |u_1-u_2|^2)(T_2 + \gamma_2 |u_1-u_2|^2) \\ \leq (1+\varepsilon) (\bar{T}_{12} +\gamma |u_1-u_2|^2) (\bar{T}_{21} + \tilde{\gamma}|u_1-u_2|^2)
\end{align*}
This is equivalent to
\begin{align*}
\varepsilon T_1 \bar{T}_{21} + \varepsilon T_1 \tilde{\gamma} |u_1-u_2|^2 + \varepsilon \gamma_1 |u_1-u_2|^2 \bar{T}_{21} + \varepsilon \gamma_1 \tilde{\gamma} |u_1-u_2|^4 + \bar{T}_{12} T_2 \\+ \bar{T}_{12} \gamma_2 |u_1-u_2|^2  + \gamma |u_1-u_2|^2 T_2+ \gamma \gamma_2 |u_1-u_2|^4 \\ \leq (1+\varepsilon) (\bar{T}_{12} \bar{T}_{21} + \bar{T}_{12} \tilde{\gamma} |u_1-u_2|^2 + \gamma |u_1-u_2|^2 \bar{T}_{21} + \gamma \tilde{\gamma} |u_1-u_2|^4
\end{align*}
This is true if we have separately 
\begin{align}
\varepsilon T_1 \bar{T}_{21} + \bar{T}_{12} T_2 \leq (1+\varepsilon) \bar{T}_{12} \bar{T}_{21}
\end{align}
\begin{align}
(\varepsilon \gamma_1 \tilde{\gamma} + \gamma \gamma_2) |u_1-u_2|^4 \leq (1+ \varepsilon) \gamma \tilde{\gamma} |u_1-u_2|^4
\end{align}
\begin{align}
\begin{split}
\varepsilon T_1 \tilde{\gamma} |u_1-u_2|^2 + \varepsilon \gamma_1 |u_1-u_2|^2 \bar{T}_{21} + \bar{T}_{12} \gamma_2 |u_1-u_2|^2 + \gamma |u_1-u_2|^2 T_2 \\ \leq (1+\varepsilon)  \bar{T}_{12} \tilde{\gamma} |u_1-u_2|^2 + (1+\varepsilon) \gamma |u_1-u_2|^2 \bar{T}_{21}
\end{split}
\end{align}
These three inequalities are satisfied according to lemmas \ref{inequ2}, \ref{inequ3}, \ref{inequ4}. Then the last four terms in \eqref{I} can be estimated by zero from above. So we obtain
\begin{align*}
I \leq - \frac{T_{12}}{m_1} c_{12} n_2 \int_{\mathbb{R}^d} \frac{M_{12}^2}{f_1} \left(\frac{\nabla_v f_1 M_{12}- \nabla_v M_{12} f_1}{M_{12}^2} \right)^2 dv\\ - \frac{T_{21}}{m_2} c_{21} n_1 \int_{\mathbb{R}^d} \frac{M_{21}^2}{f_2} \left(\frac{\nabla_v f_2 M_{21}- \nabla_v M_{21} f_2}{M_{21}^2} \right)^2 dv
\\ =  - \frac{T_{12}}{m_1} c_{12} n_2 \int_{\mathbb{R}^d} f_1 \left| \frac{M_{12}}{f_1} \nabla_v \left( \frac{f_1}{M_{12}} \right) \right|^2 dv - \frac{T_{21}}{m_2} c_{21} n_1 \int_{\mathbb{R}^d} f_2 \left| \frac{M_{21}}{f_2} \nabla_v \left( \frac{f_2}{M_{21}} \right) \right|^2 dv
\\ =  - \frac{T_{12}}{m_1} c_{12} n_2 \int_{\mathbb{R}^d} f_1 \left| \nabla_v \ln \frac{f_1}{M_{12}} \right|^2 dv - \frac{T_{21}}{m_2} c_{21} n_1 \int_{\mathbb{R}^d} f_2 \left| \nabla_v \ln \frac{f_2}{M_{21}} \right|^2 dv \leq 0
\end{align*}
with equality if and only if $f_1=M_{12}$ and $f_2=M_{21}$. This means the equality is characterized by two Maxwell distributions. In addition, if we compute the mean velocities of these expressions, we get in case of equality $u_1=u_{12}=\delta u_1+(1-\delta) u_2$ which leads to $u_1=u_2$. Similar, for the temperatures, we obtain $T_1=T_2$.
\end{proof}
Define the total entropy $H(f_1,f_2) = \int_{\mathbb{R}^d} (f_1 \ln f_1 + f_2 \ln f_2) dv$. We can compute 
$$ \partial_t H(f_1,f_2) + \nabla_x \cdot \int_{\mathbb{R}^d} ( f_1 \ln f_1 + f_2 \ln f_2 ) v dv  = S(f_1,f_2),$$ by multiplying the Fokker-Planck equation for the species $1$ by $\ln f_1$, the Fokker-Planck  equation for the species $2$ by $\ln f_2$ and integrating the sum with respect to $v$. \\

 \begin{corollary}[Entropy inequality for mixtures]
Assume $f_1, f_2 >0$. Assume  a fast enough decay of $f_1, f_2$ to zero for $v\rightarrow \infty$.
Assume relationship \eqref{coll}, the conditions \eqref{density}, \eqref{convexvel}, \eqref{veloc}, \eqref{contemp} and \eqref{temp} with $\alpha, \delta \neq 1$, the positivity of the temperatures \eqref{gamma} and the assumptions on the free parameters \eqref{asspar}. \textcolor{black}{Then} we have the following entropy inequality
$$
\partial_t \left(\int_{\mathbb{R}^d}   f_1 \ln f_1  dv + \int_{\mathbb{R}^d} f_2 \ln f_2 dv \right) + \nabla_x \cdot \left(\int_{\mathbb{R}^d}  v f_1 \ln f_1  dv + \int_{\mathbb{R}^d} v f_2 \ln f_2 dv \right) \leq 0,
$$
with equality if and only if $f_1$ and $f_2$ are Maxwell distributions with equal bulk velocity and temperature. Moreover at equilibrium the interspecies Maxwellians $M_{12}$ and $M_{21}$ satisfy $u_{12} = u_2=u_1= u_{21}$ and $T_{12} = T_2=T_1=T_{21}$.
\end{corollary}
We now explicitly specify the global equilibrium. \\

\begin{theorem}[Equilibrium]
Assume $f_1, f_2 >0$.
Assume relationship \eqref{coll}, the conditions \eqref{density}, \eqref{convexvel}, \eqref{veloc}, \eqref{contemp} and \eqref{temp} and the positivity of the temperatures \eqref{gamma}.
Then $Q^{FP}_{11}(f_1,f_1)+Q^{FP}_{12}(f_1,f_2)=0$ and $Q^{FP}_{22}(f_2,f_2)+Q^{FP}_{21}(f_2,f_1)=0$, if and only if $f_1$ and $f_2$ are Maxwell distributions with equal mean velocity and temperature.
\end{theorem}
\begin{proof}
If $Q^{FP}_{11}(f_1,f_1)+Q^{FP}_{12}(f_1,f_2)=0$ and $Q^{FP}_{22}(f_2,f_2)+Q^{FP}_{21}(f_2,f_1)=0$, then $\ln f_1 ~ Q^{FP}_{11}(f_1,f_1)+\ln f_1 ~ Q^{FP}_{12}(f_1,f_2)+\ln f_2 ~ Q^{FP}_{22}(f_2,f_2)+\ln f_2 ~ Q^{FP}_{21}(f_2,f_1)=0$ and so we have equality in the H-theorem.
\end{proof}

\textcolor{black}{\section{The \textit{N}-species case}
\label{secN}
 The two-species case can be  extended to a system of $N$-species that undergo binary collisions. We consider the $N$-species equation,
\begin{align}
\partial_t f_i + v \cdot \nabla_x f_i = \sum_{j=1}^N c_{ij} \text{ div} (\nabla_v (\frac{T_{ij}}{m_i} f_i) +  (v-u_{ij})f_i ) \quad i=1,...,N.
\end{align}
The quantity $c_{ii} $ is the friction constant concerning the interactions of particles of species $i$ with itself whereas $c_{ij} $ is the friction constant concerning the interactions of particles of species $i$ with species $j$, with $i,j=1,...,N, ~ i \neq j$. 
We only have terms of this form and not terms containing indices of more than two species because we consider only binary interactions.
For fixed $i,j\in \{1,\dots,N\}$ the  velocities $u_{ij}$ and $T_{ij}$ will be determined as follows.
The single species velocities and temperatures $u_{ii}, T_{ii}$ and $u_{jj}, T_{jj}$  will be determined  such that they satisfy theorem \ref{constraints_intra}.
The quantities $u_{ij}, T_{ij}$ and $u_{ji},T_{ji}$ will be determined such that we obtain conservation of mass of each species and conservation of total momentum and total energy in interactions between these two species, i.e.,
\begin{align}
\begin{split}
&\int_{\mathbb{R}^d} m_i v Q^{FP}_{ij}(f_i,f_j) dv +
\int_{\mathbb{R}^d} m_j v Q^{FP}_{ji}(f_j,f_i) dv = 0, 
\\
&\int_{\mathbb{R}^d} \frac{m_i}{2} |v|^2 Q^{FP}_{ij}(f_i,f_j) dv +
\int_{\mathbb{R}^d} \frac{m_j}{2} |v|^2 Q^{FP}_{ji}(f_j,f_i) dv = 0,\end{split}
\end{align}
as a straight-forward generalization of theorem \ref{theomom} and theorem \ref{theoen} with indices $i$ and $j$ instead of $1$ and $2$. Note that the free parameters $\alpha, \delta, \gamma$ can be different for each $(i,j)$, so they should be replaced by $\alpha_{ij}, \delta_{ij}, \gamma_{ij}$. 
All the proofs concerning existence and uniqueness of the target Maxwellians and the H-Theorem can be proven exactly in the same way as for two species. For the total entropy $H(f_1,...,f_N) = \int (f_1 \ln f_1  + \dots + f_N \ln f_N) dv$ we obtain
\begin{align}
\partial_t \left(  H(f_1,..., f_N) \right) + \nabla_x \cdot \left(\int  v (f_1 \ln f_1 + \dots + f_N \ln f_N) dv \right) \leq 0.
\end{align}
}

\section{Conclusions}
In this paper, a general Fokker-Planck model for gas mixtures is presented. It extends the one-species Fokker-Planck model to the setting of gas mixtures. This provides a model for gas mixtures which can be used to simulate plasma flows and is more efficient than the full Boltzmann or Landau-Fokker-Planck model.

The model proposed here is presented in the two-species case. It describes the time evolution of both species and the equation of each species contains a sum of two interaction operators, one describing the interaction of particles of the species with itself, one describing the interaction with particles of the other species. The model is constructed such that it satisfies conservation of mass, momentum and energy in the intra-species interactions, and conservation of mass, total momentum and total energy in the interactions with the other species. We provided sufficient conditions under which the model satisfies positivity of all temperatures and an H-Theorem. Moreover, we derived macroscopic equations and computed the exchange of momentum and energy which is transferred from one species to the other species. This provides the possibility to fix the exchange terms for example with the exchange terms of the Boltzmann equation. 

As a future project, we plan to extend this model to the setting of polyatomic molecules. In this setting it is possible that the particles can have degrees of freedom in internal energy in addition to the translational degrees of freedom.

\backmatter

%
%
%

\bmhead{Acknowledgements}

Marlies Pirner was funded by the Deutsche Forschungsgemeinschaft (DFG, German Research Foundation) under Germany’s Excellence Strategy EXC 2044-390685587, Mathematics Münster: Dynamics–Geometry–Structure, by the Alexander von Humboldt foundation and the German Science Foundation DFG (grant no. PI 1501/2-1).



\section*{Declarations}


\begin{itemize}
\item Competing Interests: Not applicable
\item Data availability: The authors declare that the data supporting the findings of this study are available within this paper.
\end{itemize}

\noindent

\bigskip

\bibliography{sn-bibliography}

\end{document}